\documentclass[a4paper,12pt]{article}
\usepackage[english]{babel}
\usepackage[latin1]{inputenc}
\usepackage{amsmath}
\usepackage{amssymb}

\usepackage{amsthm}
\usepackage{latexsym}
\usepackage{verbatim}
\usepackage{enumitem}
\usepackage{pinlabel}
\usepackage{graphics}
\usepackage{amsfonts}
\usepackage{subcaption}
\usepackage{booktabs, multirow, siunitx, caption}
\usepackage{graphicx}
\usepackage{epstopdf}
\usepackage{pdflscape}
\usepackage{multicol}
\usepackage[margin=0.5in]{geometry}
\newcommand{\pl}{{\operatorname{pl}}}
\newcommand{\stick}{{\operatorname{stick}}}
\newcommand{\cro}{{\operatorname{cr}}}

\newtheorem{lem}{Lemma}

\newtheorem{cor}{Corollary}

\newtheorem{prob}{Problem}
\newtheorem{obs}{Observation}

\def\b{\beta}

\def\o*{\overline{*}}
\def\u*{\underline{*}}

\theoremstyle{definition}
\newtheorem{example}{Example}
\newtheorem{defn}{Definition}

\usepackage{thmtools} 
\usepackage{thm-restate}
\usepackage{lipsum}

\declaretheorem[name=Theorem,numberwithin=section]{thm}
\begin{document}

\title{Planar stick indices of some knotted graphs}

\author{Tirasan Khandhawit, Puttipong Pongtanapaisan, Athibadee Wasun}

\date{\today}

\maketitle
\begin{abstract}
Two isomorphic graphs can have inequivalent spatial embeddings in 3-space. In this way, an isomorphism class of graphs contains many spatial graph types. A common way to measure the complexity of a spatial graph type is to count the minimum number of straight sticks needed for its construction in 3-space. In this paper, we give estimates of this quantity by enumerating stick diagrams in a plane. In particular, we compute the planar stick indices of knotted graphs with low crossing numbers. We also show that if a bouquet graph or a theta-curve has the property that its proper subgraphs are all trivial, then the planar stick index must be at least seven.

\end{abstract}
\section{Introduction}
For many decades, scientists have been interested in synthesizing molecules in the shape of various graphs (see \cite{kim2018coordination,li2011metallosupramolecular}, for instance). It is reasonable to model an atom as a vertex and a bond between atoms as a straight stick connecting the vertices, resulting in a rigid piecewise-linear presentation of the knotted graph. Thus, to construct a knotted molecule with certain level of complexity, it is beneficial to understand the minimum number of straight sticks one needs to construct a knotted graph type. Such a quantity called the \textbf{stick number} is easy to define, but can be difficult to compute.

When the graph type is the cycle graph, the stick number is a complexity measure for knots. In fact, it is not known what the precise values of the stick numbers of the majority of eight crossing knots are at the time of this writing and not many lower estimates exist. Huh was able to show that any seven points in general position of 3-space constitute at most three
heptagonal figure-eight knots \cite{huh2012knotted} because knots with seven sticks are well-understood. The lack of classification of knots with stick number nine makes the task of determining the types of quantities of knots in a linear embedding of a complete graph challenging.

In this paper, we investigate a planar analog of the stick index. The concept was introduced by Adams et al. in \cite{adams2011planar} for knots, and we extend the idea to spatial graphs. The following task is still not easy and may require the aid of computers, but it gives an alternate perspective to analyze the stick numbers.
\begin{prob}\label{question:classify}
  Enumerate planar stick diagrams of knots with eight sticks.
\end{prob}
 That is, if a knot from the list of 8-crossing knots with unknown stick number does not belong to the list from Problem \ref{question:classify}, then the stick index is precisely 10.

Stick number of more general graph types were studied in \cite{lee2017stick}, where the authors bound the stick number from above by a function of crossing number, number of vertices/edges, and number of bouquet cut components. In this paper, we characterize some properties of knotted objects based on their planar stick indices. In particular, we are interested in \textbf{ravels}, which are in spatial graphs that are not planar, but contains no knots or links. A ravel is intriguing because its entanglement cannot be determined by looking at a particular cycle. Our results can be summarized as follows:

\begin{enumerate}
    \item The planar stick index of knots \(3_1\) and \(5_1\) is 5. The planar stick index of knots \(4_1, 5_2, 6_1, 6_2,\) and \(7_4\) is 6. The planar stick index of knots other than these knots is at least seven.
    \item  If $G$ is a ravel (topological) bouquet graph or a $\theta$-curve, then the planar stick index of $G$ is at least seven.
\end{enumerate}

In each case above, we will give specific spatial graphs that realize the lower bounds. Compared to other spatial graphs, bouquet graphs and $\theta$-curves have appeared in numerous biological objects \cite{dabrowski2019theta,castle2008ravels,ceniceros2023rna,li2011metallosupramolecular}. In \cite{flapan2019stick}, the authors used the stick numbers to compute the probability that a random linear embedding of $K_{3,3}$ in a cube is in M\"{o}bius form. Perhaps, the computations in this paper can be combined with other results to provide insight into knotting probability.

\subsection*{Organization}
This paper is organized as follows. In Section \ref{section:background}, we gather the definitions that readers can refer back to. To make the definitions more understandable, we demonstrate them in examples. Some of these examples will be the ones that provide sharp bounds in the main theorems. In Section \ref{section:relations}, we relate the planar stick index to other known invariants. In Section \ref{section:cycle}, we compute planar stick index of cycle graphs, which is the setting of knots. Sections \ref{section:bouquets} and \ref{section:theta} contain results on planar stick indices of bouquet graphs and $\theta$-curves, respectively.

\section{Background}\label{section:background}

A \textbf{graph} is a 1-dimensional complex made up of vertices and edges. We study embeddings of graphs in the 3-space $\mathbb{R}^3$ up to ambient isotopy. This is a generalization of knot theory since the underlying graph type of a knot can be taken to be a cycle graph. Due to a result of Kauffman \cite{kauffman1989invariants}, these graph embeddings can be studied combinatorially and diagrammatically.
\begin{defn}
A \textbf{spatial graph diagram} is an immersion of a graph $G$ in the plane such that the image of a vertex does not coincide with the image of a different vertex or a point in the interior of an edge. We also require the edges to intersect transversely, where each double point is decorated as a classical crossing. Referring to Figure \ref{fig:greid}, a \textbf{rigid-vertex spatial graph} is an equivalence class of spatial graph diagrams modulo Reidemeister moves R1-R5. A \textbf{topological spatial graph} (or simply \textbf{a spatial graph}) is an equivalence class of spatial graph diagrams modulo Reidemeister moves R1-R6 of Figure \ref{fig:greid}.
\end{defn}
\begin{figure}[ht!]
    \centering
    \includegraphics[scale=0.2]{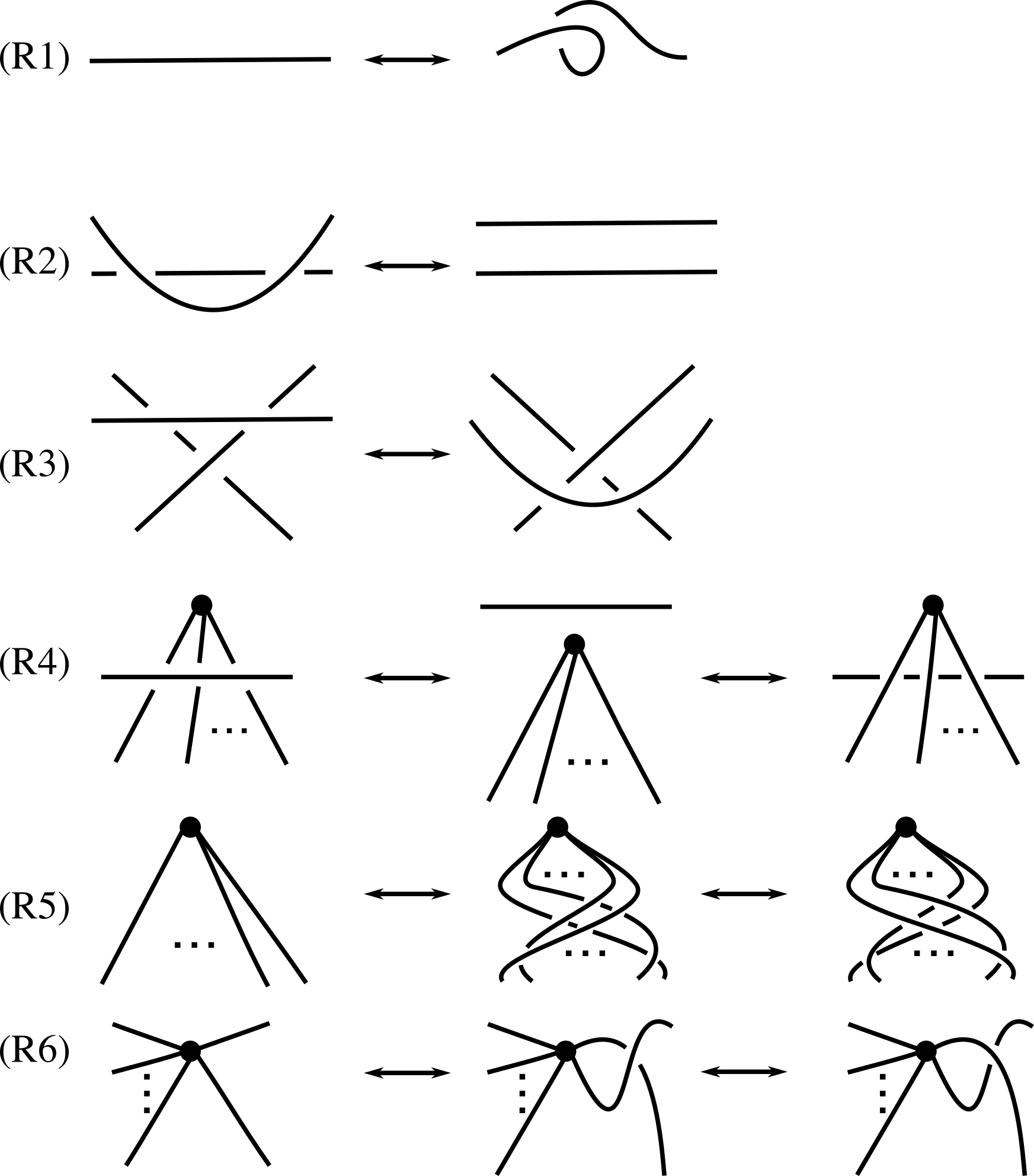}
    \caption{Reidemeister moves for graphs.}
    \label{fig:greid}
\end{figure}
\begin{defn}
A \textbf{planar stick diagram} of a spatial graph is a spatial graph diagram (with or without crossing information) where each edge is linear. The \textbf{planar stick index} $\pl[G]$ of a spatial graph is the smallest number of edges in any planar stick diagram of $G$.
\end{defn}

This is the planar analog of the following more studied quantity.
\begin{defn}
The 3D \textbf{stick number} $\stick[G]$ of $G$ is the minimum number of straight sticks one needs to construct a knotted graph type in $\mathbb{R}^3$.
\end{defn}
\begin{example}\label{Kinoexample}
    The left of Figure \ref{fig:kino} shows a projection of a piecewise linear embedding of the Kinoshita $\theta$-curve taken from \cite{huh2009stick}. To realize the $\theta$-curve in 3-space, the ``bend" in the stick that is not a part of the star-shape cycle is truly necessary. On the other hand, the right of Figure \ref{fig:kino} shows that on the diagram level, that bending is not needing and the planar stick index is at most seven. In this paper, we show that it is exactly seven.
\end{example}
\begin{figure}[ht!]
    \centering
    \includegraphics[scale=0.4]{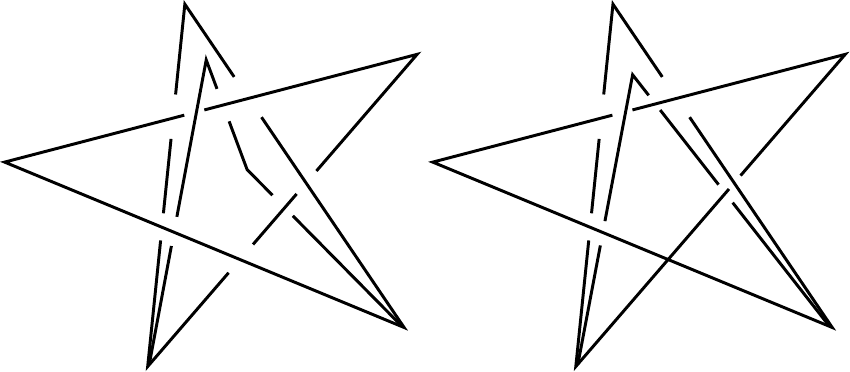}
    \caption{The Kinoshita $\theta$-curve has planar stick index seven, but 3D stick number eight.}
    \label{fig:kino}
\end{figure}
As mentioned before, the authors in \cite{lee2017stick} exhibit a relationship between the following invariant and stick numbers.
\begin{defn}
    The \textbf{crossing number} $\cro[G]$ is the minimum number of crossings over all diagrams of $G$. 
\end{defn}
Next, we remind the readers of a knot invariant that is related to the planar stick index of cycle graphs.
\begin{defn}
    The \textbf{bridge number} $\b[K]$ of a knot is the minimum number of local maxima with respect to the standard height function in $\mathbb{R}^2$ over all diagrams of $K$. 
\end{defn}
Lastly, there is a type of embeddings that is interesting because we cannot determine the nontrivaility of the graph based on the collection of links contained in the embedding. 
\begin{defn}
  A nontrivial spatial graph having the property that no collection
of disjoint cycles is a non-trivial link is called a \textbf{ravel}.
\end{defn}

In \cite{castle2008ravels}, the authors also considered the notion of \textbf{$n$-ravels} which are non-trivial
entanglements around a vertex with no knots and links formed by the mutual weaving
of $n$ edges emerging from a single vertex. Bouquet graphs and $\theta$-curves arise from closing up ravels in various ways (see Figure \ref{fig:vertexclosure}).

\begin{figure}[ht!]
    \centering
    \includegraphics[scale=0.3]{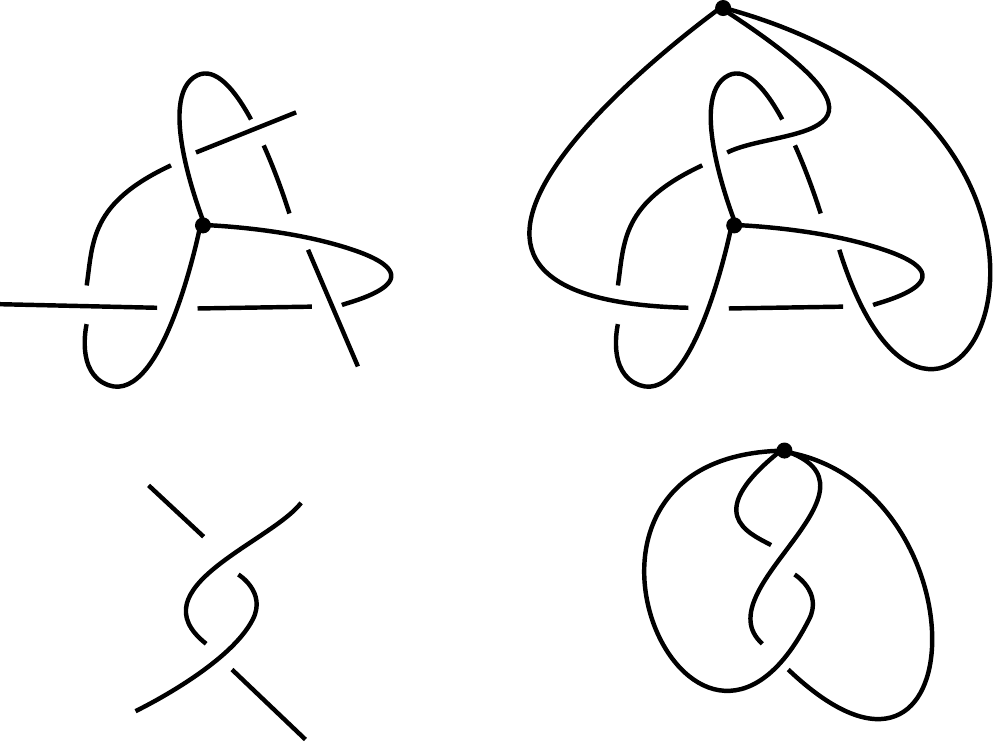}
    \caption{(Top) The vertex closure of a 3-ravel is a $\theta$-curve. (Bottom) The vertex closure of a 2-string tangle is a bouquet graph.}
    \label{fig:vertexclosure}
\end{figure}
\subsection{Tri-colorability}
Later in the paper, we need a way to verify that certain spatial graphs are topologically nontrivial. This is a harder task than verifying rigid vertex inequivalence. An easily computable topological invariant in knot theory is called \textbf{tri-colorabilility} (see Section 1.5 of the famous knot book \cite{adams1994knot} by Colin Adams). Ishii showed in \cite{ishii1997color} that the idea extends into spatial graphs whose vertices all have even degrees. More precisely, a spatial graph diagram is \textbf{tricolorable} if each of the strands in the projection can be colored one of three different colors, so that at each crossing, either three different colors come together or all the same color comes together. We further require that at least two of the colors are used and all arcs adjacent to the same vertex are the same color.
\begin{example}\label{RavelExample}
    The spatial bouquet graph diagram in Figure \ref{fig:3color} is tricolorable. Therefore, it is topologically nontrivial as the trivial graph is not tricolorable.
\end{example}
\begin{figure}[ht!]
    \centering
    \includegraphics[scale=0.5]{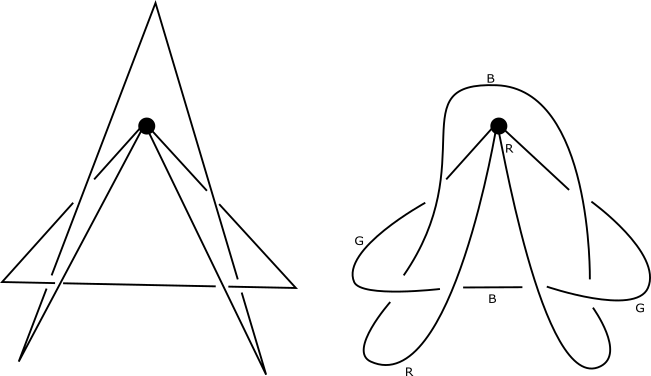}
    \caption{(Left) A planar stick diagram of $6^k_{19}$ with six sticks. (Right) A smoothed out version of the image on the left, so that the tricolorability can be demonstrated.}
    \label{fig:3color}
\end{figure}
\section{Relations to other invariants}\label{section:relations}
Let $\stick[G]$ denote the 3D stick number of a graph.  In \cite{adams2011planar}, the authors showed that the planar stick index is at most one less than the stick number. The following statement is a generalization to the spatial graph setting. 
\begin{thm}
    If $G$ is a spatial graph such that embeddings realizing $\stick[G]$ contain a cycle of length at least four, then $\pl[G] \leq \stick[G]-1.$
\end{thm}
\begin{proof}
    Let $c$ be the cycle in the graph made up of $n$ edges $e_1,e_2,\cdots, e_n$, where $n\geq 4$. After perturbing the embedding slightly preserving the linearity of edges, there is an edge such that projecting onto a plane normal to that edge gives a planar stick diagram of $\stick[G]-1$ sticks. The inequality follows from this observation.
\end{proof}

The following result shows a relationship between $\pl[G]$ and $\cro[G]$. We now set up the notations. Let $D$ be a planar stick diagram $D$ realizing $\pl[G]$. Let $n_d$ be the number of edges of $D$ each with the property that $d$ other edges are connected to it. For instance, an edge adjacent to the 4-valent vertex in Figure \ref{fig:3color} has $d=4,$ since a vertex bounding that edge has three other edges adjacent to it and the other vertex bounding the edge has one vertex adjacent to it.
\begin{thm}\label{thm:crossing}
    Let $G$ be a spatial graph. Then, $\cro[G] \leq \dfrac{\sum_{d} n_d(\pl[G]-d-1)}{2}$, where the sum ranges over all possible numbers arising as sums of the degrees of the two vertices bounding an edge. 
\end{thm}
\begin{proof}
Let $D$ be a diagram realizing $\pl[G] = n.$ Each edge $e$ cannot cross itself and any of the edges attached to the vertices bounding $e.$ Since there are $n_d$ edges each connecting two vertices summing to degree $d,$ we get the $n_d(\pl[G]-d-1)$ term. The number 2 in the denominator comes from the fact that an intersection at a crossing involves two edges, so we have counted twice for each intersection.
\end{proof}

For example, if $G$ is a bouquet graph with $7$ sticks, then $n_4=4$ and $n_2=3.$ By Theorem \ref{thm:crossing}, the crossing number is at most $\frac{4(7-5)+3(7-3)}{2}=10$. Performing a case-by-case analysis in Section \ref{section:bouquets} will allow us to bring this upper bound down to 7.
\section{Cycle graphs}\label{section:cycle}
This case coincides with planar stick number of knots. Nicholson has computed $\pl[G]$ for virtual knots with real crossing numbers at most five \cite{nicholson2011piecewise}. In this paper, we computed some planar stick number for some knots with real crossing numbers at most nine.

Our main strategy is to enumerate planar stick diagram with at most 6 sticks. In particular, we consider placement of sticks on the plane first and then adding all possible crossing information. 
We begin by introducing the following result obtained by Adams et al.

\begin{thm} \label{thm:thm1}
\cite{adams2011planar} Let \(\cro[K]\) be the crossing number of a knot \(K\) and \(\pl[K]\) the planar stick index of knot \(K\). Then
\[ \frac{3+\sqrt{9 + 8\,\cro[K]}}{2} \leq \pl[K]. \]
\end{thm}

In other words, $\cro[K]\leq \frac{1}{2}\pl[K](\pl[K]-3).$ This can also be obtained from Theorem \ref{thm:crossing} by noticing that $d$ is always 2 and $n_d$ is always $\pl[G]$ for any planar stick diagram. As a result, a stick planar diagram for a nontrivial knot has at least 5 sticks. 



\subsection{Planar diagrams with 5 sticks}\label{Subsection:ProofTheorem4.2}
\begin{thm} \label{thm:thm2}
The only knots with \(\pl[K] = 5\) are $3_1$ and $5_1$. 
\end{thm}
\begin{proof}

We will construct a planar diagram by placing five sticks on the plane. 
We shall assume that there is a stick that intersects with two other sticks, otherwise there will be at most $\frac{5 \times 1}{2} = 2.5$ crossings and a diagram can only give a trivial knot.

Without loss of generality, we designate one of these sticks as the first stick and position it horizontally. Subsequently, we position the second stick in an upward direction at the right endpoint. 

There are two main cases based on the intersection points of the 3rd and 4th sticks with the first stick. We will use the notation $(3 \, 4)$ to represent the case where the intersection with the 3rd stick occurs on the left side. In this particular case, the diagram is obtained as in Figure~\ref{fig:34d}.

\begin{figure}[ht!]
    \centering
    \includegraphics[scale=0.5]{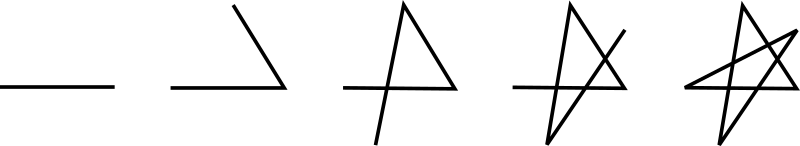}
    \caption{The diagram of the case $(3 \, 4)$.}
    \label{fig:34d}
\end{figure}
\noindent 
We may assume that the 4th should intersect with the 2nd stick to create more crossings. Alternatively, the diagram in which 4th stick and the 5th stick do not intersect the 2nd can be incorporated into the previous case using type 2 Reidemeister move as follows.
\begin{figure}[ht!]
    \centering
    \includegraphics[scale=0.5]{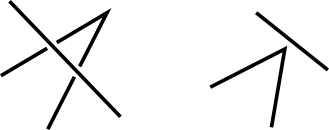}
    \caption{Type 2 Reidemeister move.}
    \label{fig:5type2}
\end{figure}

Together with the case $(4 \, 3)$, we obtain all possible planar diagrams with 5 sticks as in Figure~\ref{fig:5sticksall}.  After adding crossing information, we obtain the knots $3_1$ and $5_1$.  
\begin{figure}[ht!]
    \centering
    \includegraphics[scale=0.7]{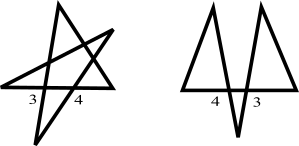}
    \caption{Planar diagrams with 5 sticks.}
    \label{fig:5sticksall}
\end{figure}

\end{proof} 

\subsection{Planar diagrams with 6 sticks}\label{subsection:ProofTheorem4.4}

As a result of Theorem~\ref{thm:thm2}, it follows that all knots except for $3_1$ and $5_1$ have  planar stick index of at least 6. We will proceed to analyze planar diagrams involving 6 sticks in a manner similar to the case with 5 sticks.

\begin{thm} \label{thm:thm4}
The only knots with \(\pl[K] = 6\) are \(4_1, 5_2, 6_1, 6_2\) and \(7_4\). 
\end{thm}
\begin{proof}

Since we can find planar diagrams with 6 sticks for the knots $4_1$ and $5_2$, we will focus on a planar diagram with at least 6 crossings.

We can assume that there is a stick that intersects with 3 other sticks. Otherwise, each stick would need to intersect with exactly 2 other sticks to create 6 crossings, resulting in 2  intersection pairings for the 6 sticks. For each intersection pairing, it can be verified that drawing the corresponding diagram is impossible.

We select a stick that intersects with 3 other sticks and designate it as the first stick, positioning it horizontally. Next, we place the second stick in an upward direction at the right endpoint. We will examine cases based on the positions of the intersection points of the 3rd, 4th, and 5th sticks with the first stick. There are $3! = 6$ cases arising from permutation of 3 letters. Similarly, we use the notation $(3 \, 4 \, 5)$ to denote the case where the intersection points are arranged on the first stick from left to right. The diagrams are shown in the following table. 


\begin{minipage}{1.0\textwidth}
    \centering
    \vspace{5mm}
    \label{tab:table2}
    \begin{tabular}{ccccccc}
        \toprule
         {Case} & {Diagram} &   & {Case} & {Diagram} \\
        \midrule
       (\(3\,4\,5\)) &	\parbox[c]{7em}{\includegraphics[scale=0.25]{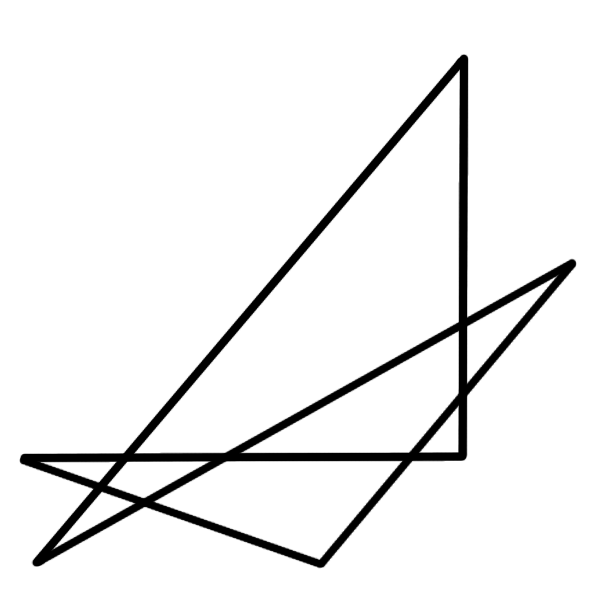}} &&	
       (\(4\,3\,5\)) &	\parbox[c]{7em}{\includegraphics[scale=0.25]{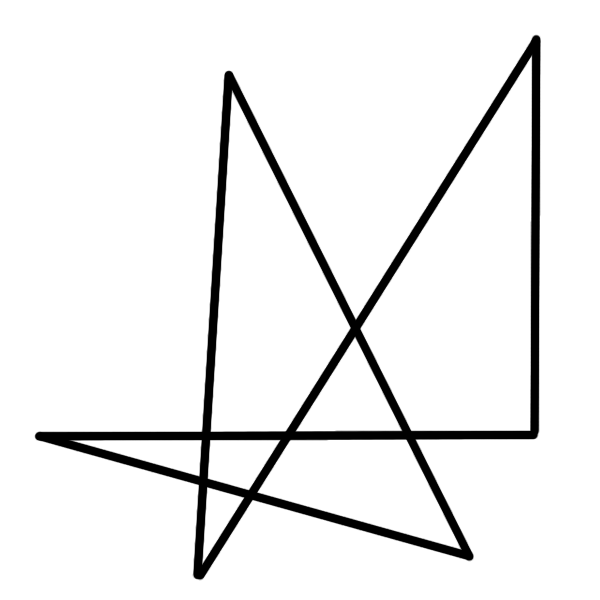}}	\\
       (\(3\,5\,4\)) &	\parbox[c]{7em}{\includegraphics[scale=0.25]{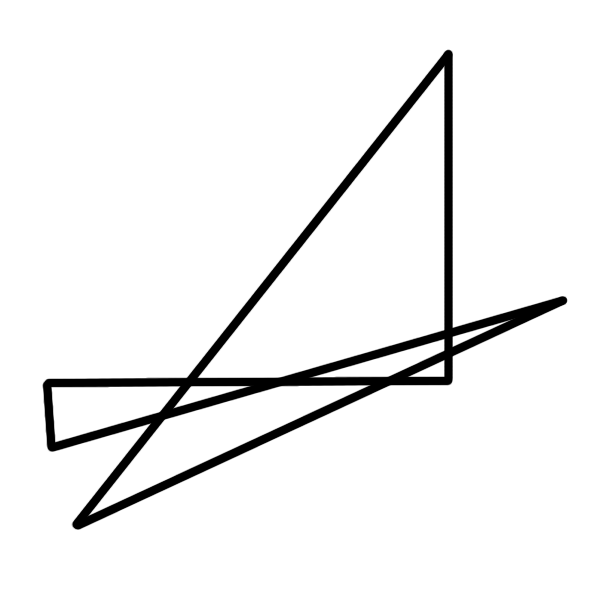}} &&
       (\(4\,5\,3\)) &	\parbox[c]{7em}{\includegraphics[scale=0.25]{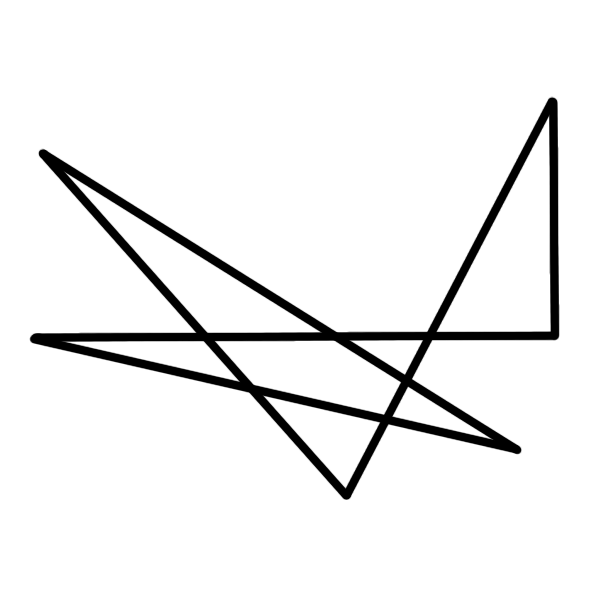}}	\\
       (\(5\,3\,4\)) &	\parbox[c]{7em}{\includegraphics[scale=0.25]{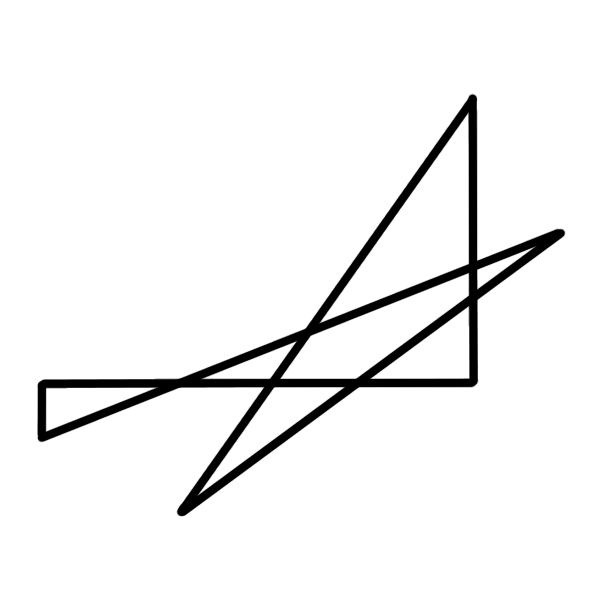}} &&
       (\(5\,4\,3\)) &	\parbox[c]{7em}{\includegraphics[scale=0.25]{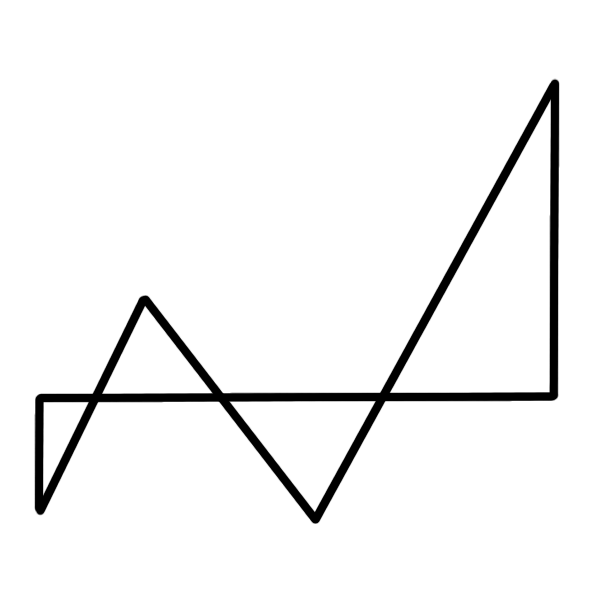}}	\\
      \\	
      \bottomrule
    \end{tabular}
\end{minipage} \\ \\

There is a symmetry among the diagrams; specifically, we can rotate a diagram by \(180^{\circ}\) and reverse the order, so that the 6th stick becomes the 2nd stick. With this operation, the diagram $(4\,3\,5)$ is equivalent to the diagram $(3\,5\,4)$ and the diagram $(4\,5\,3)$ is equivalent to the diagram $(5\,3\,4)$. Note that the diagrams $(3\,4\,5)$ and $(5\,4\,3)$ are preserved under this operation.

The diagram $(4\,5\,3)$ can be reduced to a 5-stick diagram via type 1 Reidemeister move. The diagram $(5\,4\,3)$ only gives a trivial knot diagram. Hence, we only have 2 cases to consider.

\textbf{Case $(3\,4\,5)$}: With crossings added, we obtain the following knots: $7_4, 6_1, 5_2,$ and $4_1$ 
\begin{figure}[htbp]
    \centering
    \includegraphics[scale=0.75]{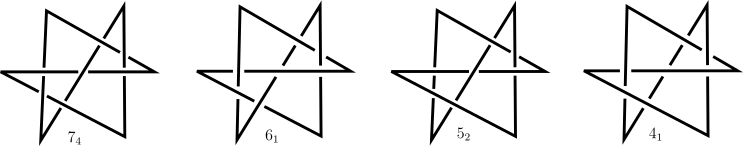}
    \label{fig:345a}
\end{figure}

\textbf{Case $(4\,3\,5)$}: With crossings added, we obtain the following knots: $6_2, 5_2,$ and $4_1$ 
\begin{figure}[htbp]
    \centering
    \includegraphics[scale=0.75]{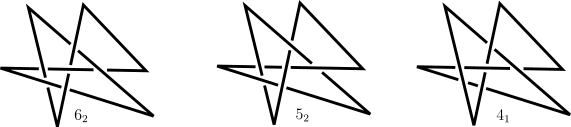}
    \label{fig:345a}
\end{figure}

\end{proof}

\subsection{Identifying planar stick indices of small knots}

Theorem~\ref{thm:thm2} and Theorem~\ref{thm:thm4} have a straightforward consequence.

\begin{cor}
The planar stick index of knots other than \(3_1, 4_1, 5_1, 5_2, 6_1, 6_2,\) and \(7_4\) must be at least \(7\).
\end{cor}

By experiment, we try to find 7-stick planar diagrams of the rest knots. As a result, we can determine planar stick index of knots up to 7 crossings. We also find many knots with planar stick index 7. See the table in Appendix \ref{appendix}.


\subsubsection{More knots with 7 planar sticks}

There are more knot types possessing planar stick diagrams with seven sticks displayed in the table. For instance, knots $7_1,7_3,$ and $ 8_2$ have the same planar stick diagrams in the table differing only in the crossing information.

We use SnapPy \cite{SnapPy} to obtain PD code for each planar stick diagram in the table, then use a Python code to permute the PD codes. This results in all possible knots possessing a fix planar stick diagram. The code is available upon request. The result is summarized as follows. We will only apply the code to the diagrams with at least 8 crossings since we know the planar stick numbers of all 7 crossing knots. We follow the convention on KnotInfo.

\begin{enumerate}
\item The diagram for $8_{19}$ in the table gives $8_{10},8_{16},8_{21}, 10_{100}, 10_{124},  10_{125},10_{141},10_{143}, 10_{155}$.

These knots come from permuting the PD code 
\begin{flalign*} 
& [(2,9,3,10),(3,11,4,10),(20,13,1,14),(7,12,8,13),(1,15,2,14), && \\ &(4,17,5,18),(8,16,9,15),(11,16,12,17),(5,19,6,18),(6,19,7,20)] .  &&
\end{flalign*}
\begin{figure}[htbp]
    \centering
    \includegraphics[scale=0.75]{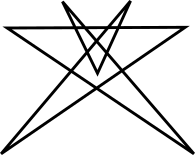}
    \label{fig:78_19}
\end{figure}

    \item The shadow of the diagram for $9_{26}$ in the table gives knots 
$0_1, 3_1, 4_1, 5_1, 5_2, 6_1, 6_2, 6_3, 7_1, \\ 7_2,  7_3, 7_5, 7_6, 7_7, 8_2, 8_4, 8_6, 8_7, 8_8, 8_9, 8_{14}, 9_1, 9_6, 9_7, 9_9, 9_{20}, 9_{23}, 9_{26}, 10_9, 10_{22}, 10_{32}, K11a306
$ .

These knots come from permuting the PD code 
\begin{flalign*} 
& [(21,9,22,8),(18,9,19,10),(11,16,12,17),(7,21,8,20),(19,7,20,6),
(5,15,6,14), && \\ &(13,5,14,4), (22,18,1,17),(15,3,16,2),(3,13,4,12),(1,10,2,11)]  . &&
\end{flalign*}
\begin{figure}[htbp]
    \centering
    \includegraphics[scale=0.75]{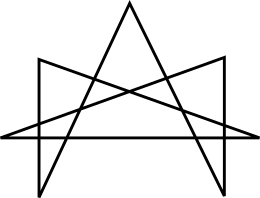}
    \label{fig:79_26}
\end{figure}

\item Unlike the case with 6 sticks, it is possible to draw a diagram that attains maximal crossing number $\frac{7 \times 4}{2}= 14$. One of the diagrams is shown below with the PD code:
\begin{flalign*} 
& [(15,24,16,25), (12,22,13,21), (22,4,23,3), (18,9,19,10), (19,28,20,1), (16,8,17,7),  && \\&(17,27,18,26), (20,12,21,11), (13,3,14,2), (25,7,26,6), (23,4,24,5), (1,11,2,10), && \\&(14,5,15,6), (27,8,28,9)] .  &&
\end{flalign*}
\begin{figure}[htbp]
    \centering
    \includegraphics[scale=0.75]{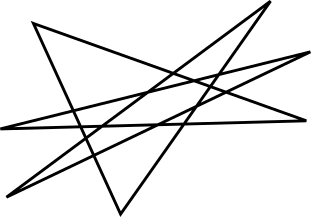}
    \label{fig:714}
\end{figure}
We obtain the following knots from permuting the PD code:
\begin{flalign*} 
& 0_1, 3_1, 4_1, 5_1, 5_2, 6_2, 6_3, 7_1, 7_3, 7_5, 8_2, 8_5, 8_7, 8_9, 8_{10}, 8_{16}, 8_{17}, 8_{18}, 8_{19}, 8_{20},  8_{21}, 9_9, 9_{16}, 10_9, &&\\ & 10_{17}, 10_{48}, 10_{64}, 10_{79}, 10_{82}, 10_{85}, 10_{91}, 10_{94}, 10_{99}, 10_{100}, 10_{104}, 10_{106}, 10_{109}, 10_{112}, 10_{116},&&\\& 10_{118},  10_{124}, 10_{127}, 10_{139}, 10_{141}, 10_{143}, 10_{148}, 10_{149}, 10_{152}, 10_{155}, 10_{157}, 10_{159}, 10_{161}, &&
\end{flalign*}
K12n242, K12n344, K12n467, K12n468, K12n571, K12n708, K12n721, K12n725, K12n747, K12n748, K12n749, K12n751, K12n767, K12n821, K12n829, K12n830, K12n831, K12a819, K12a864, K12a1002, K12a1013, K12a1209, K12a1211, K12a1219, K12a1221, K12a1226, K12a1230, K12a1248, K12a1253, K14n20437, K14n21192, K14n21881, K14n21882, K14n21884, K14n22339, K14n22344, K14n23999, K14n24169, K14n24767, K14n27039, K14n27117, K14n27120, K14n27123, K14n27133, K14n27154, K14a19475.
\end{enumerate}

\section{Bouquet graphs}\label{section:bouquets}
We often make use of the enumeration result by Oyamaguchi \cite{oyamaguchi2015enumeration}. We remind the readers that this enumeration is up to rigid vertex equivalence. We begin by computing planar stick indices of bouquet graphs with low crossing numbers.
\subsection{Computations for low crossing bouquet graphs}

The first lemma characterizes planar stick diagrams with six sticks.
\begin{lem}
There are three types of planar stick diagrams of a bouquet graph such that $\pl[D]=6$. \label{lem:bouquetpl6}
\end{lem}
\begin{proof}
Consider a triangle $\Delta$, which is one of the two loops of the bouquet graph $G$ made up of three edges $\{e_1,e_2,e_3\}$. Call the 4-valent vertex $x$ and let $e_3$ be the edge not connected to $x$. The triangle $\Delta$ cuts the plane into two regions. We call the region with bounded area the inside and the complementary region the outside of $\Delta$. Let $e_4,e_5$ be the remaining two edges connecting to $x$. 

\textbf{Case 1}: Suppose that as $e_4$ and $e_5$ emerge from $x$, they are on the inside of $\Delta$.

\textbf{Subcase 1.1} If $e_4$ and $e_5$ are disjoint from $\Delta,$ we get the planar stick diagram of Type B0.

\textbf{Subcase 1.2} Suppose that one of $e_4$ or $e_5$ intersects $\Delta$. Since each edge is straight, this forces such an edge to intersect $e_3$. All the possibilities lead to the planar stick diagram of Type B1.

\textbf{Subcase 1.3} Suppose that both of $e_4$ or $e_5$ intersects $\Delta$. Since each edge is straight, both edges intersect $e_3$. However, all the ways of adding the crossing information to this subcase give the same bouquet graphs as the results of adding crossing information to the planar stick diagram of Type B1.

\textbf{Case 2}: Suppose as $e_4$ and $e_5$ emerge from $x$, they are on the outside of $\Delta$. This forces $e_4$ and $e_5$ to be disjoint from $e_1$ and $e_2.$ If $e_6$ is disjoint from $\Delta,$ we get the planar stick diagram of Type B0. The other possibility left is $e_6$ intersects $\Delta$ twice and we get the planar stick diagram of Type B1.

\textbf{Case 3}: As $e_4$ and $e_5$ exits $x$, they are on opposite sides of $\Delta.$ Say $e_4$ is inside $\Delta$ as it leaves $x.$

\textbf{Subcase 3.1} Suppose the whole edge $e_4$ is completely inside $\Delta.$ Then, we get a Type B2 planar stick diagram because the sixth edge $e_6$ has to intersect $\Delta$ once to go inside of $\Delta$ to connect with $e_4.$

\textbf{Subcase 3.2} Suppose a part of $e_4$ lies outside of $\Delta.$ Then, we can get a Type B2 planar stick diagram if $e_6$ is disjoint from $\Delta$ or we get a Type B3 planar stick diagram if $e_6$ intersects $e_3$ twice and either $e_1$ or $e_2$ once.
\end{proof}

\begin{figure}
     \centering
     \begin{subfigure}[b]{0.3\textwidth}
         \centering
         \includegraphics[width=\textwidth]{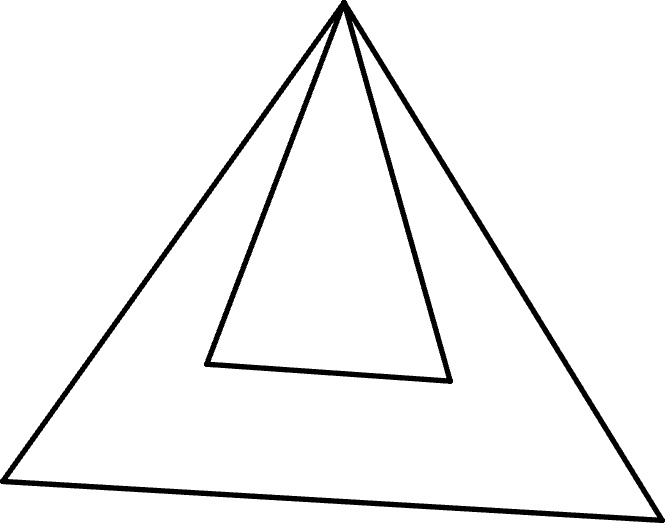}
         \caption{A type B0 planar stick diagram}
         \label{fig:b0}
     \end{subfigure}
     \begin{subfigure}[b]{0.3\textwidth}
         \centering
         \includegraphics[width=\textwidth]{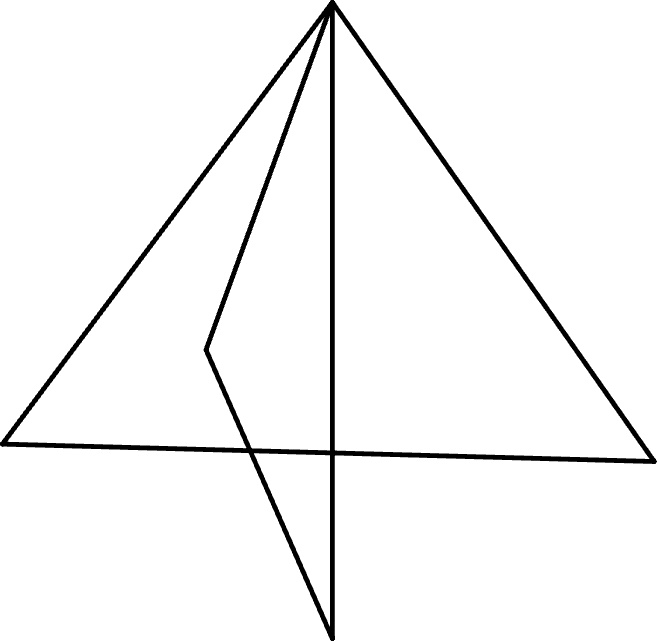}
         \caption{A type B1 planar stick diagram}
         \label{fig:b2}
     \end{subfigure}
        \label{fig:b0yb2}
\end{figure}

\begin{figure}
     \centering
     \begin{subfigure}[b]{0.3\textwidth}
         \centering
         \includegraphics[width=\textwidth]{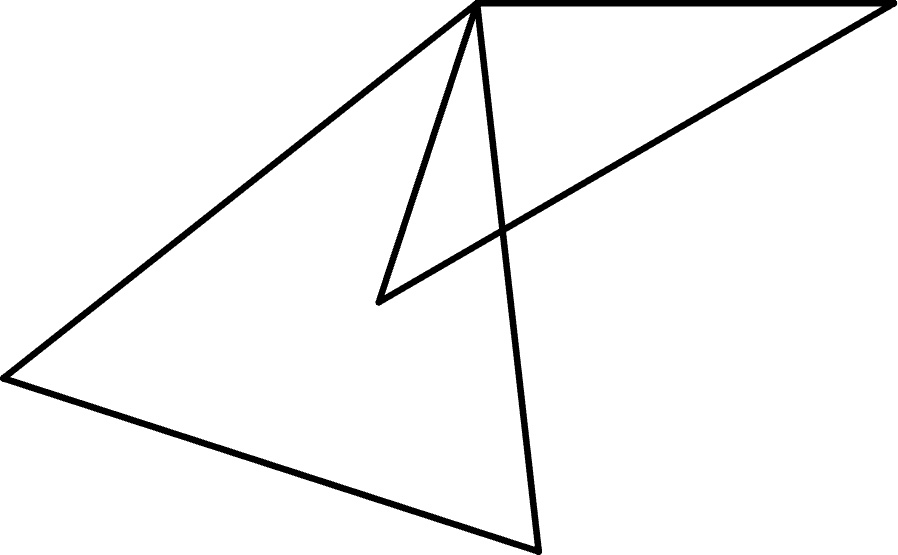}
         \caption{A type B2 planar stick diagram}
         \label{fig:b1}
     \end{subfigure}
     \begin{subfigure}[b]{0.22\textwidth}
         \centering
         \includegraphics[width=\textwidth]{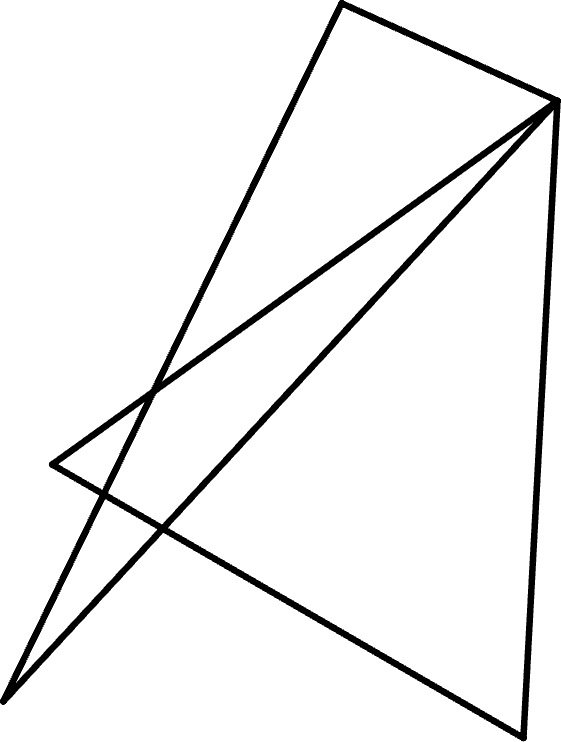}
         \caption{A type B3 planar stick diagram}
         \label{fig:b3}
     \end{subfigure}
        \label{fig:b1yb3}
\end{figure}

\begin{figure}
    \centering
    \includegraphics[width=0.5\textwidth]{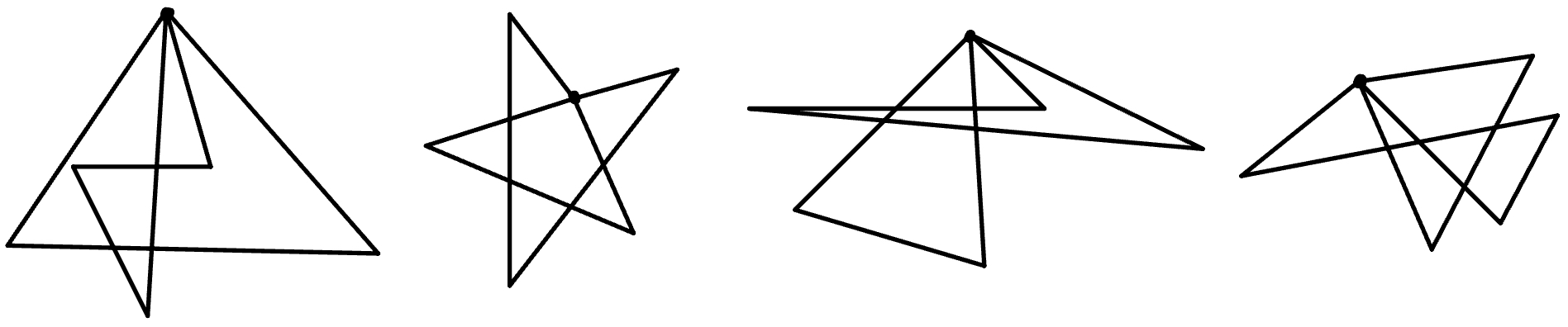}
    \caption{From left to right, planar stick diagrams of $3_1^k,4_1^k,4_2^k,4_1^l$ with seven sticks.}
    \label{fig:4crossings}
\end{figure}
\begin{figure}
    \centering
    \includegraphics[width=0.5\textwidth]{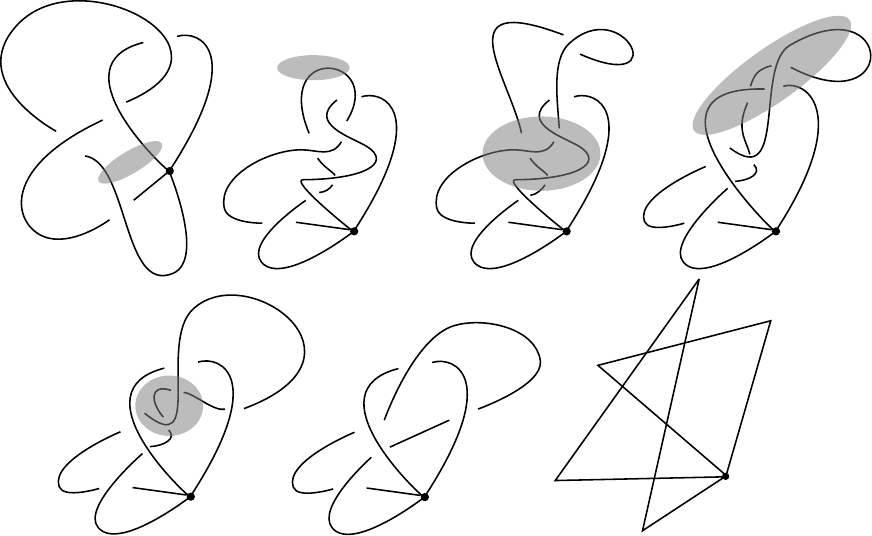}
    \caption{The graph $4_3^k$ has planar stick index 7. The top left image depicts the diagram in Oyamaguchi's paper. The lower right picture is a planar stick diagram. The images in between show how to use Reidemeister moves to go to the spatial graph diagram corresponding to the stick diagram. If a gray oval intersects the diagram in $k$ strands, it means that the next diagram differs from the current diagram by a Reidemeister (R$k$) move, where $k \in \{1,2,3\}$ (see Figure \ref{fig:greid}). Here, the last image of the top row is followed by the first image of the bottom row.}
    \label{fig:4k3}
\end{figure}
We will often make use of the straightforward observation.

\begin{obs}\label{lem:intersect}
    Suppose that $D$ is a planar stick diagram with 7 sticks that realizes $\pl[G]$. Let $\Delta$ be the triangle made up of three sticks. For an edge connecting $e$ two vertices $v$ and $w,$ the intersection numbers of $e$ with $\Delta$ is $|e\cap \Delta|=1$ if $v$ and $w$ lie on opposite sides of $\Delta$. If $v$ and $w$ both lie inside of $\Delta$, then $|e\cap \Delta|=0$. Lastly, if $v$ and $w$ both lie outside of $\Delta$, then $|e\cap \Delta|=0$ or 2.
\end{obs}
\begin{thm}
   If $\pl[G]=7$, then $\cro[G]\leq 7$. Furthermore, planar stick diagrams with 7 sticks are among diagrams C1-C10, C9$'$,C10$'$, D, and E. \label{thm:ravelbouquet}
\end{thm}
\begin{proof}
  Take such a planar stick diagram $D$. This means that a cycle of the bouquet graph is made up of three sticks, and the other cycle is made up of four sticks.

    Consider a triangle $\Delta$, which is one of the two loops of the bouquet graph $G$ made up of three edges $\{e_1,e_2,e_3\}$. Call the 4-valent vertex $x$ and let $e_3$ be the edge not connected to $x$. The triangle $\Delta$ cuts the plane into two regions. We call the region with bounded area the inside and the complementary region the outside of $\Delta$. Let $e_4,e_5$ be the remaining two edges connecting to $x$. 

\textbf{Case 1}: Suppose that as $e_4$ and $e_5$ emerge from $x$, they are on the inside of $\Delta$.

\textbf{Subcase 1.1} Suppose $e_4$ and $e_5$ are disjoint from $\Delta.$

By Observation \ref{lem:intersect}, $e_4$ and $e_5$ do not contribute to the crossing number of $D$. The edge $e_6$ and $e_7$ can intersect $\Delta$ at most once each. But $e_6$ has to be connected to $e_7$, so they intersect the same edge of $\Delta.$ The cycle made up of 4 edges can have at most one self-intersection. Thus, the crossing number is at most three in this case. 

\textbf{Subcase 1.2} Suppose that one of $\{ e_4,e_5\}$ intersects $\Delta$.
Say, $e_5$ intersects $\Delta$ and $e_6$ connects to $e_5.$ We want to create diagrams with as many crossings as possible. This forces $e_6$ to intersect $\Delta$ twice. After this, there is a unique choice for $e_7.$ All the possible diagrams are shown in Figure \ref{fig:c1c2c3}. In this case, the crossing number is at most five.

\textbf{Subcase 1.3} Suppose that $e_4,e_5$ both intersect $\Delta$.

We can break into further cases. For the configurations where both $e_6$ and $e_7$ only hit $e_3$, we get the planar stick diagrams as in Figure \ref{fig:onlyhite3}. For the configurations where $e_6$ hits two edges of $\Delta$, but $e_7$ is disjoint from $\Delta$, we get the planar stick diagrams in Figure \ref{fig:e6hit2e7disjoint.jpg}. Lastly, if $e_6$ and $e_7$ both hit two edges, then we get planar stick diagrams in Figure \ref{fig:e6e7bothhit2}. For all these cases, the crossing number does not exceed seven.

\textbf{Case 2}: As $e_4$ and $e_5$ exits $x$, they are both on the outside of $\Delta.$

In this case, the crossing number is at most five. To see this, note that to create the maximum number of crossings, each of $e_6,e_7$ crosses $\Delta$ twice. If $e_6$ also intersects $e_4$ then the crossing number is five (see Figure \ref{fig:typeD} for an example).

\textbf{Case 3}: As $e_4$ and $e_5$ exits $x$, they are on opposite sides of $\Delta.$ Say $e_4$ is inside $\Delta$ as it leaves $x.$

\textbf{Subcase 3.1} Suppose the whole edge $e_4$ is completely inside $\Delta.$ At most, $e_6$ can intersect $\Delta \cup e_4$ in three points. After this, there is a unique way to draw $e_7$ and the diagram has at most 4 crossings.

\textbf{Subcase 3.2} Suppose a part of $e_4$ lies outside of $\Delta.$ Say $e_6$ connects to $e_5.$ If $e_6$ ends inside $\Delta,$ then the crossing number of the bouquet graph is at most four. Otherwise, $e_6$ intersects $\Delta$ twice and $e_4$ once. This gives the diagram in Figure \ref{fig:TypeE}, which has six crossings.
\end{proof}

\begin{figure}
    \centering
    \includegraphics[width=0.5\textwidth]{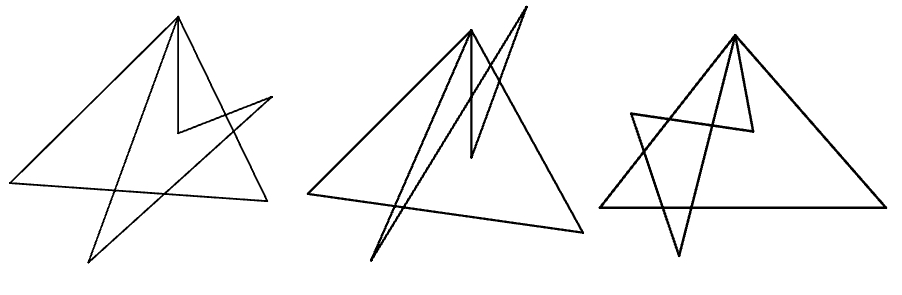}
    \caption{(Left) Type C1 diagram. (Middle) Type C2 diagram. (Right) Type C3 diagram.}
    \label{fig:c1c2c3}
\end{figure}

\begin{figure}
    \centering
    \includegraphics[width=0.4\textwidth]{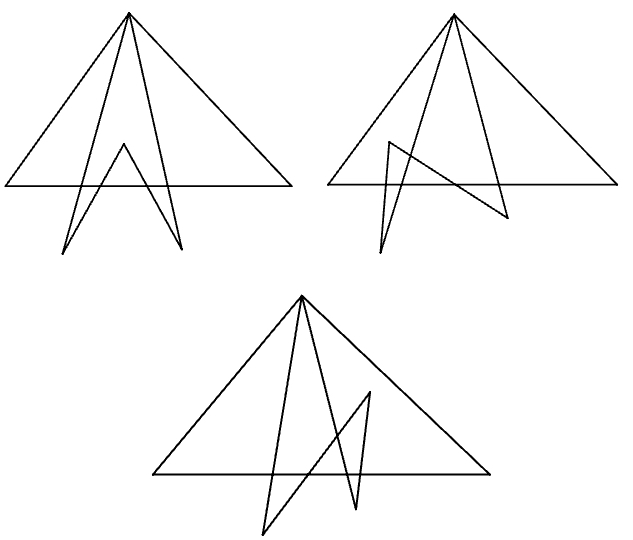}
    \caption{(Top left) Type C4. (Top right) Type C5. (Bottom) Type C6.}
    \label{fig:onlyhite3}
\end{figure}
\begin{figure}
    \centering
    \includegraphics[width=0.5\textwidth]{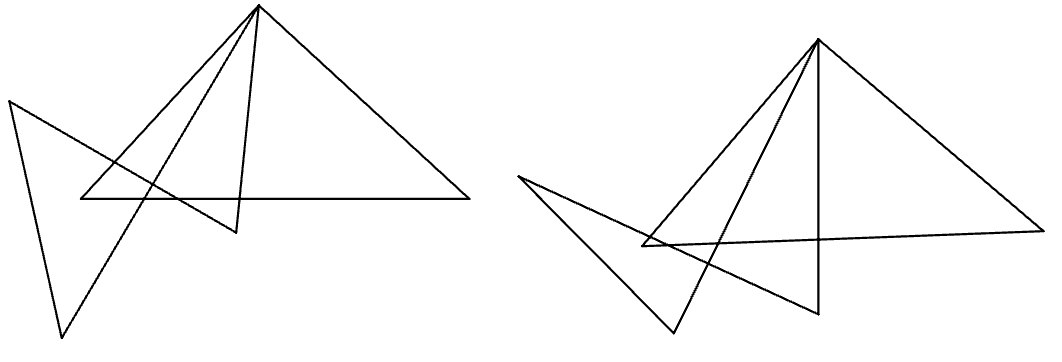}
    \caption{(Left) Type C7 diagram. (Right) Type C8 diagram.}
    \label{fig:e6hit2e7disjoint.jpg}
\end{figure}

\begin{figure}
    \centering
    \includegraphics[width=0.5\textwidth]{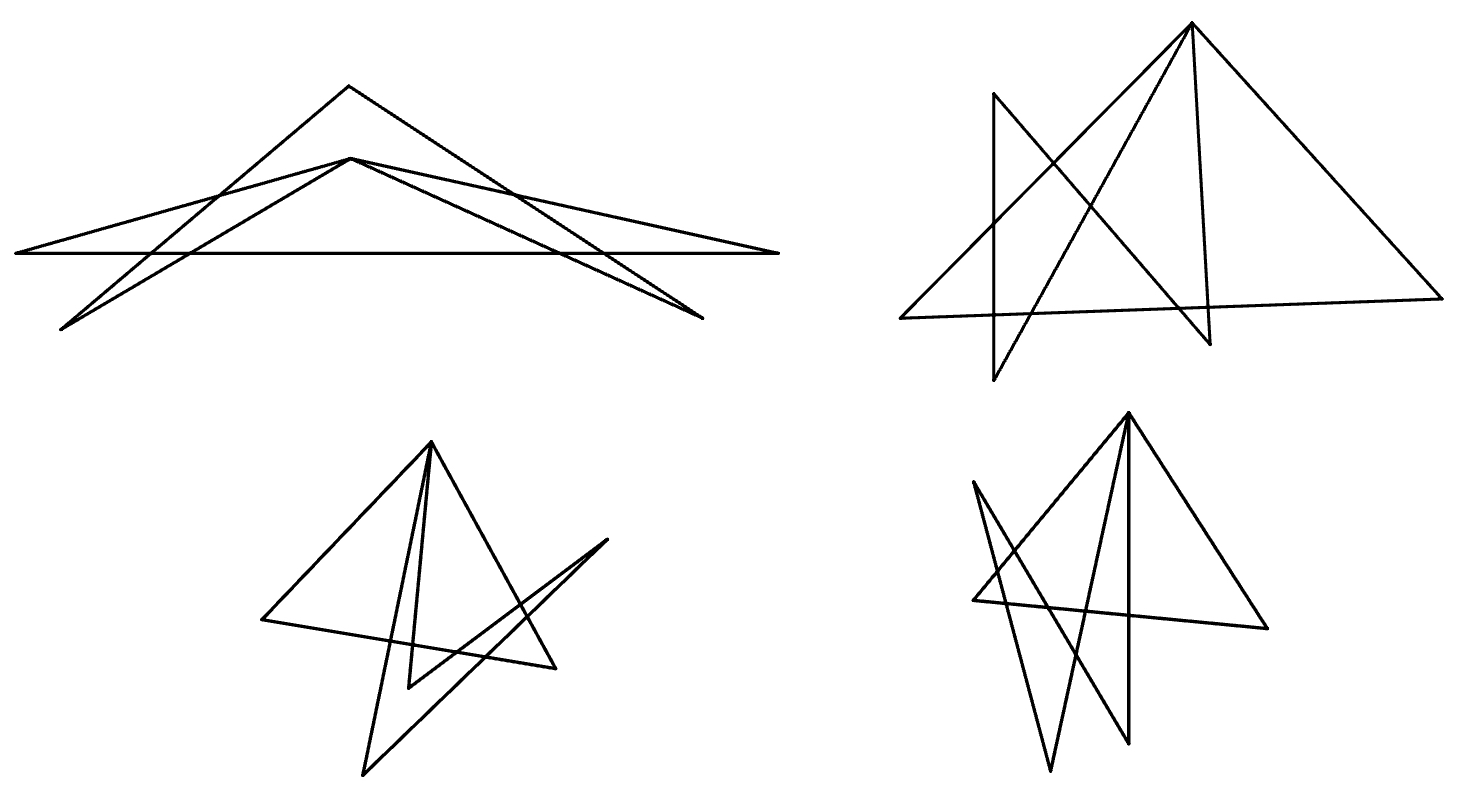}
    \caption{(Upper left) Type C9 diagram. (Upper right) Type C10 diagram. (Lower left) Type C9$'$ diagram. (Lower right) Type C10$'$ diagram.}
    \label{fig:e6e7bothhit2}
\end{figure}

\begin{figure}
    \centering
    \includegraphics[width=0.1\textwidth]{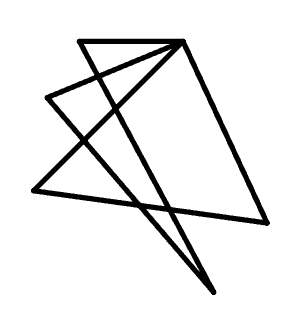}
    \caption{A type D diagram.}
    \label{fig:typeD}
\end{figure}

\begin{figure}
    \centering
    \includegraphics[width=0.2\textwidth]{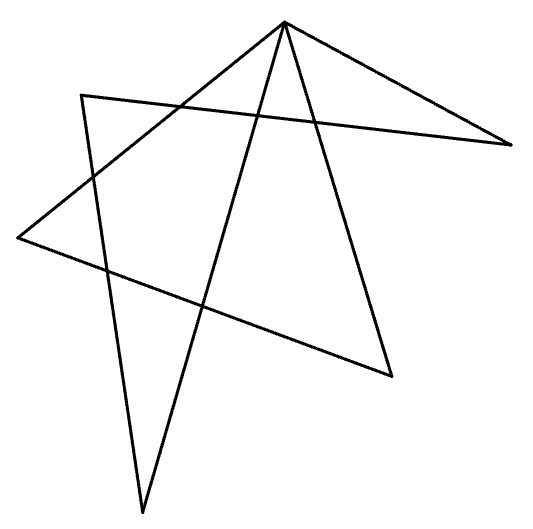}
    \caption{A type E planar stick diagram}
    \label{fig:TypeE}
\end{figure}

\begin{thm}\label{thm:lowsticks}
\begin{enumerate}
    \item The only classical bouquet rigid-vertex graphs $G$ with $\pl[G]=6$ are the graphs $0_1^k,1_1^l,2_1^k, 3_1^l$.
    \item The classical bouquet rigid-vertex graphs $3_1^k,4_1^k,4_2^k,4_3^k, 4_1^l,5_1^k, 5_3^k, 5_5^k,6^k_{19}$ have $\pl[G]=7$.
    \item  The classical bouquet rigid-vertex graphs $5_4^k,5_7^k,5_8^k$ have $\pl[G]=8$.
\end{enumerate}

\end{thm}

\begin{proof}
    For part 1 of the statement, we use Lemma \ref{lem:bouquetpl6}, where planar stick diagrams are classified. By adding all crossing information to Type B0, Type B1, Type B2, and Type B3 planar stick diagrams, we get that there are four possible bouquet graph types with planar stick index equaling six from the table in \cite{oyamaguchi2015enumeration}: $0_1^k,1_1^l,2_1^k, 3_1^l$.

    For part 2, we see from Figure \ref{fig:4crossings} and Figure \ref{fig:4k3} that $3_1^k,4_1^k,4_2^k,4_3^k, 4_1^l$ have planar stick index seven. This takes care of cases where the crossing number is at most 4. The planar stick diagram of type C3 gives $5_1^k$ after crossing information is added. The planar stick diagram of type C5 gives $5_5^k$ after crossing information is added. The planar stick diagram of type D gives $5_3^k$ after crossing information is added.

    For part 3 of the statement, we notice that $5_4^k,5_7^k,5_8^k$ each has a trefoil in it. By Theorem \ref{thm:thm2}, the trefoil knot admits a 5 stick planar diagram $D$, and admits no diagram with fewer sticks. The graphs $5_4^k,5_7^k,5_8^k$ can be obtained from $D$ by adding the trivial cycle of length three.
\end{proof}
\subsection{Ravel topological bouquet graphs}
\begin{lem}\label{notravel}
    Bouquet graphs with at most five crossings in Oyamaguchi's table are not ravels.
\end{lem}
\begin{proof}
    The statement is verified by going through the entries of the table. Note that if two adjacent edges that emerge from the vertex form a crossing right away, we can always remove that crossing by move (R6) in Figure \ref{fig:greid}. We demonstrate this process in Figure \ref{fig:simplifybouquet}, where the rightmost image is three more (R6) moves away from the trivial bouquet graph. This is the case for all entries except $5_4^k,5_7^k,$ and $5_8^k$, so that all diagrams except these three can be trivialized. For the three diagrams $5_4^k,5_7^k,$ and $5_8^k$, one of the cycles in the bouquet graph is a trefoil knot. These conclusions violate the definition of a ravel.
\end{proof}
\begin{figure}
    \centering
    \includegraphics[width=0.5\textwidth]{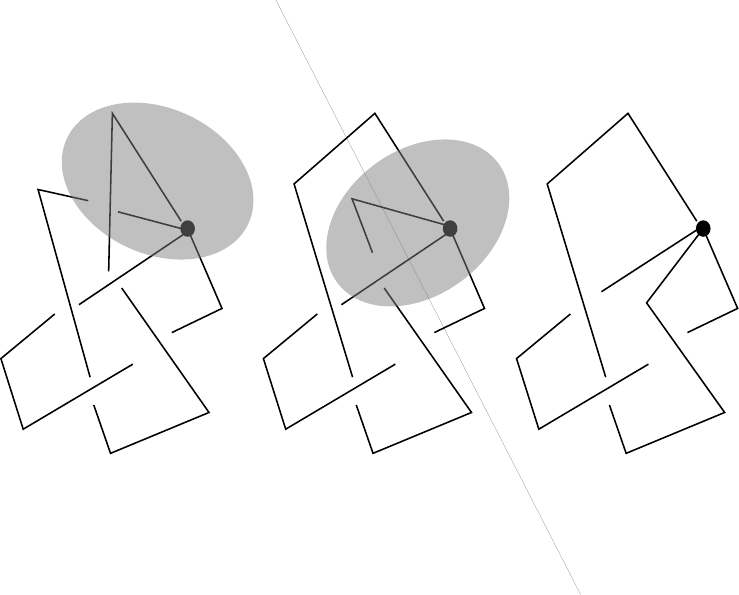}
    \caption{Simplifying a diagram of bouquet $5_3^k$ with moves (R6).}
    \label{fig:simplifybouquet}
\end{figure}

\begin{thm}
    If $G$ is a topological ravel bouquet graph, then $\pl[G]$ is at least 7. Furthermore, the bound is sharp.
\end{thm}

\begin{proof}
    By Lemma \ref{notravel}, the crossing number of $G$ is at least six. Example \ref{RavelExample} gives an example of a topological ravel bouquet graph with seven planar sticks.
\end{proof}
\section{Theta-curves}\label{section:theta}
For our computations, we use the enumeration result from Moriuchi \cite{moriuchi2009enumeration}.

\begin{thm}
   If $\theta$ is a ravel $\theta$-curve, then $\pl[\theta]\geq 7.$\label{thetaravel}
\end{thm}
\begin{proof}
    Suppose that we have a planar stick diagram $s$ with six planar sticks. By the structure of the theta curve, there must be a cycle made up of four sticks or a cycle made up of five sticks. To see this, suppose a cycle of length three (i.e. a triangle) is found. Then, there are three more edges left unconsidered. These edges form a path connecting two vertices of the triangle. The three edges plus an edge of the triangle gives a cycle of length four. 
    
    \textbf{Case 1}: There is a cycle $c$ of 4 sticks, but no cycle of 5 sticks. A cycle $c$ is then a quadrangle, which can be concave or convex. It can also intersect itself once. 

    In this case, there are 2 edges we must add from vertex $v$ to vertex $w$ to form $s$. Furthermore, there is no edge from $v$ to $w$. All ways of adding two edges in this manner result in a diagram with number of crossings strictly less than five. This is because edge of the two edges we want to add can intersect $c$ in at most two points. Since $5_1$ needs at least five crossings, $s$ is not a planar stick diagram of $5_1$.

\textbf{Case 2}: There is a cycle $c$ of 5 sticks

    We rely on the classification of cycles with 5 planar sticks done in Section \ref{Subsection:ProofTheorem4.2}. We look at the resulting theta curves obtained by adding two edges to $c$. When $c$ is a planar stick diagram with 3 crossings, then any edge we add will not increase the crossing number. Since $5_1$ needs at least five crossings, $s$ is not a planar stick diagram of $5_1$. 

    On the other hand, suppose $c$ is a planar stick diagram with 5 crossings. If we assign the crossing information so that $c$ is alternating, then our theta-curve contains a torus knot, which is a contradiction since all constituents of $5_1$ are unknots. However, if we assign the crossing information so that $c$ is not alternating, then one can reduce the resulting knot diagram by a Reidemeister II move. This means that $s$ is not a planar stick diagram of $5_1$ since it has crossing number less than 5.
\end{proof}

\begin{example}
    The Kinoshita $\theta$-curve, which is the $\theta$-curve $5_1$ in Moriuchi's table admits a planar stick diagram with seven sticks. Therefore, the bound in Theorem \ref{thetaravel} is sharp. See Example \ref{Kinoexample} for a more detailed discussion.
\end{example}
\subsection*{Acknowledgements}
The first author is supported by the Centre of Excellence
in Mathematics, the Commission on Higher Education, Thailand.
\bibliographystyle{plain}
{\small
\bibliography{refs}
}

\appendix 
\newpage
\section{Table of knots with seven planar sticks} \label{appendix}
The following table shows planar diagrams of knots with 7 sticks. For reference, we include previously known bounds from ~\cite{adams2011planar}. There are two lower bounds: $\frac{3+\sqrt{9 + 8\,\cro[K]}}{2}$ denoted by LB1 and   \( 2 \operatorname{b}[K]+1\) denoted by LB2. The upper bound  \(\pl[K] \leq \operatorname{s}[K]-1\) is denoted by UB. Note that planar stick numbers of 3-bridge knots are already known to be at least 7 and planar stick number of knots $8_{19}$ and $8_{20}$ can be determined by these bounds. \\

\noindent
\begin{minipage}{1.0\textwidth}
    \centering
    \vspace{5mm}
    \label{tab:table2a}
    \begin{tabular}{cccccccccc}
        \toprule
        {Knot}  & {LB1} & {LB2} & {UB} & {Diagram} & {Knot}  & {LB1} & {LB2} & {UB} & {Diagram} \\
        \midrule
        \(6_3\)  & 6 & 5 & 7  & \parbox[c]{7em}{\includegraphics[width=7em]{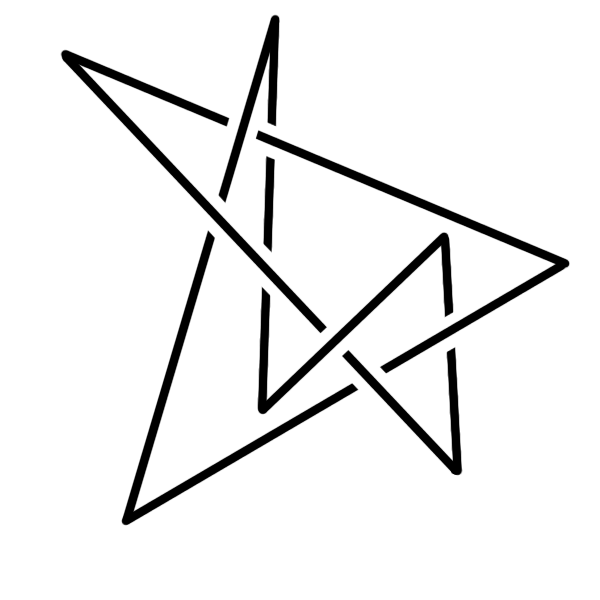}}
        & \(7_6\)  & 6 & 5 & 8  & \parbox[c]{7em}{\includegraphics[width=7em]{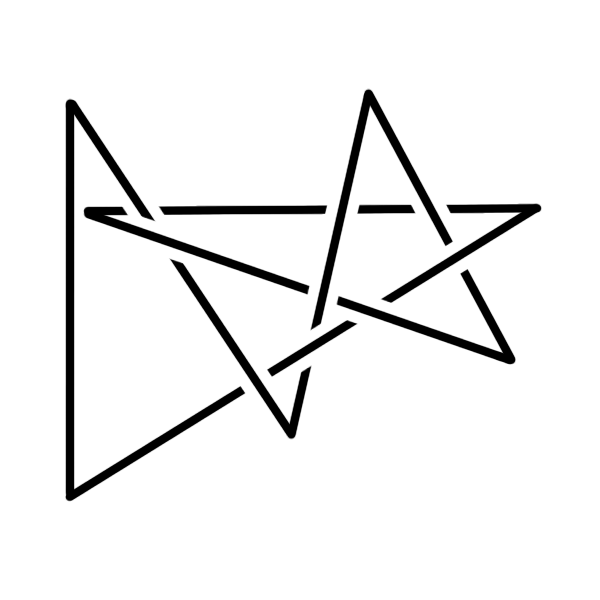}}\\
        \(7_1\)  & 6 & 5 & 8  & \parbox[c]{7em}{\includegraphics[width=7em]{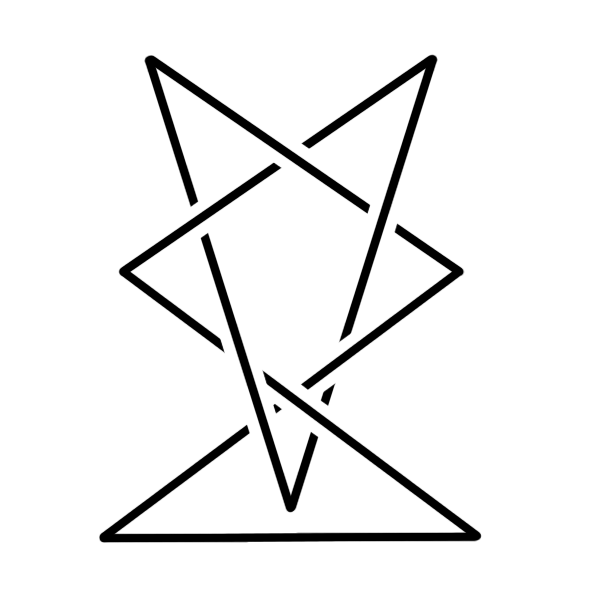}}	
        & \(7_7\)  & 6 & 5 & 8  & \parbox[c]{7em}{\includegraphics[width=7em]{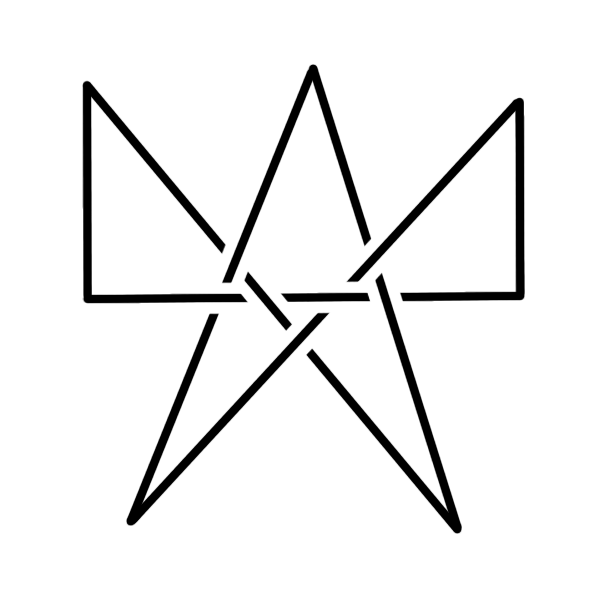}}\\
        \(7_2\)  & 6 & 5 & 8  & \parbox[c]{7em}{\includegraphics[width=7em]{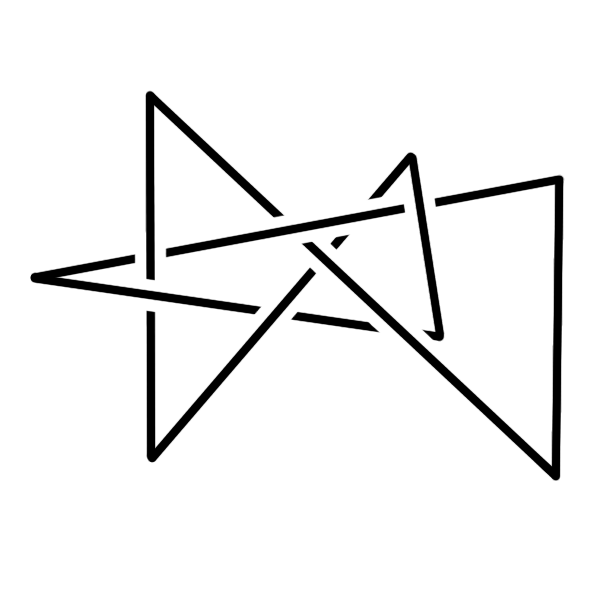}}	
        & \(8_2\)  & 6 & 5 & 9  & \parbox[c]{7em}{\includegraphics[width=7em]{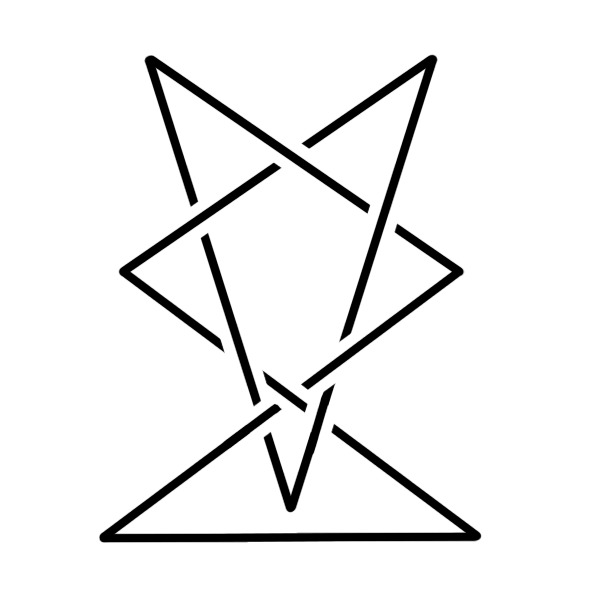}}\\ 
        \(7_3\)  & 6 & 5 & 8  & \parbox[c]{7em}{\includegraphics[width=7em]{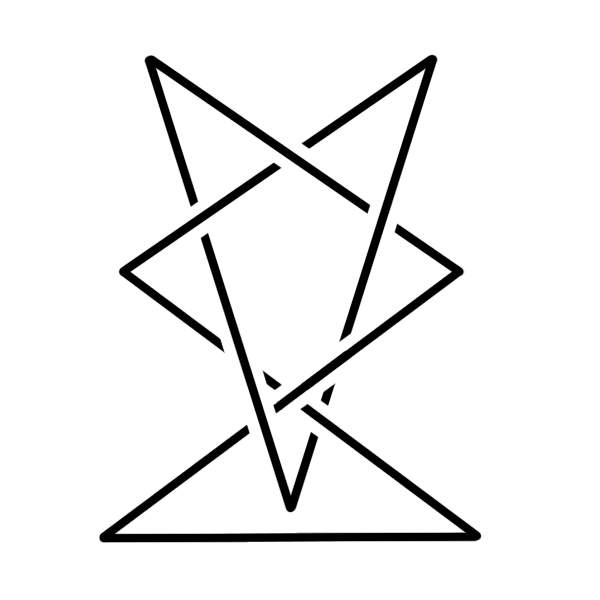}}	
        & \(8_4\)  & 6 & 5 & 9  & \parbox[c]{7em}{\includegraphics[width=7em]{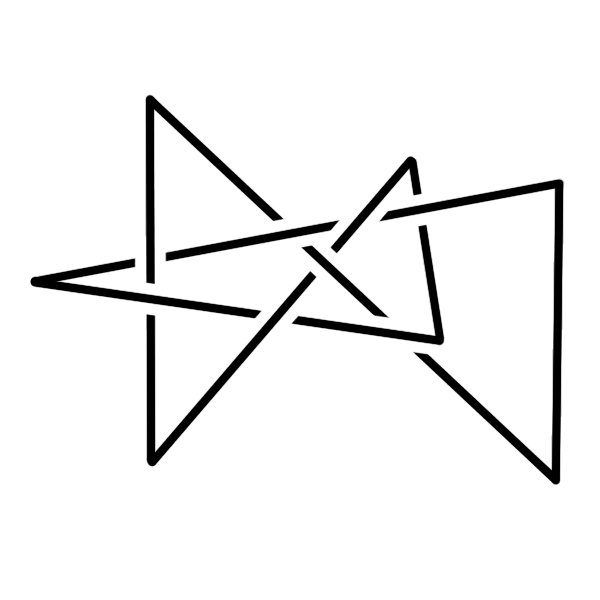}}\\
        \(7_5\)  & 6 & 5 & 8  & \parbox[c]{7em}{\includegraphics[width=7em]{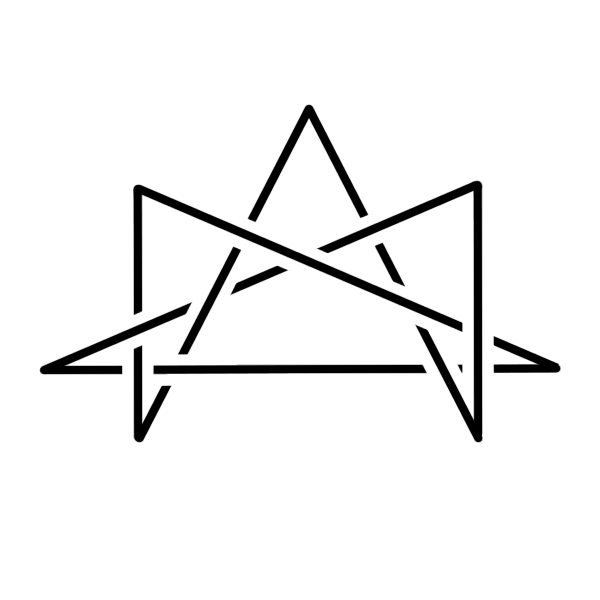}}
    & \(8_6\)  & 6 & 5 & 9  & \parbox[c]{7em}{\includegraphics[width=7em]{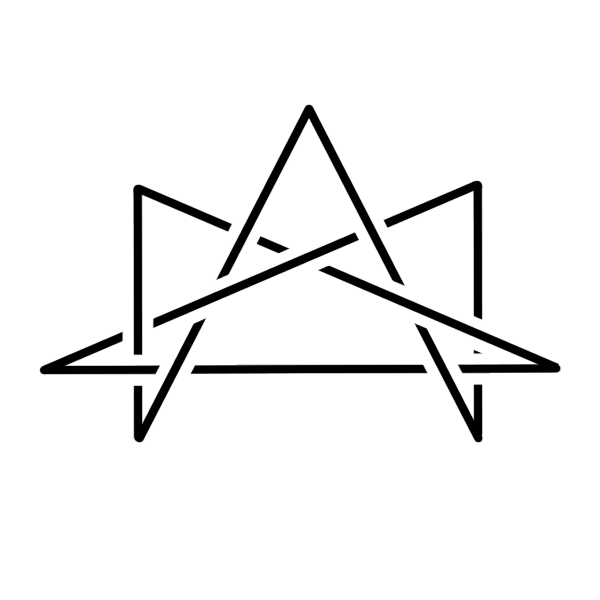}}\\

    \end{tabular}
\end{minipage}

\newpage
\noindent
\begin{minipage}{1.0\textwidth}
    \centering
    \label{tab:table2a}
    \begin{tabular}{cccccccccc}
        \toprule
        {Knot}  & {LB1} & {LB2} & {UB} & {Diagram} & {Knot}  & {LB1} & {LB2} & {UB} & {Diagram} \\
        \midrule
        \(8_7\)  & 6 & 5 & 9  & \parbox[c]{7em}{\includegraphics[width=7em]{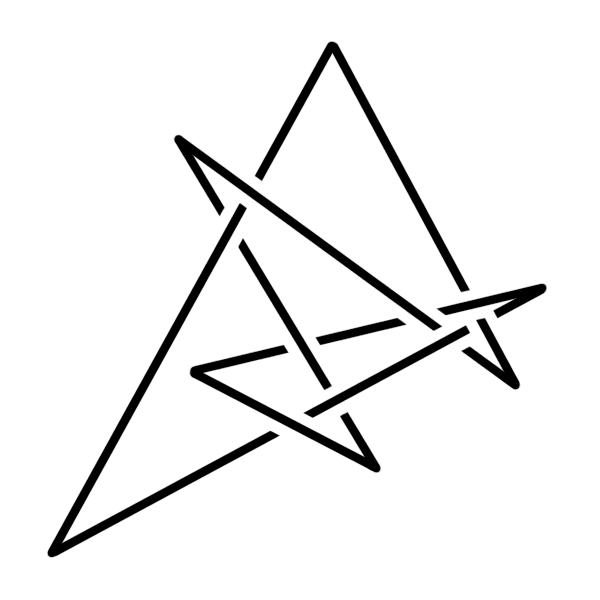}}
        & \(9_6\)  & 6 & 5 & 10  & \parbox[c]{7em}{\includegraphics[width=7em]{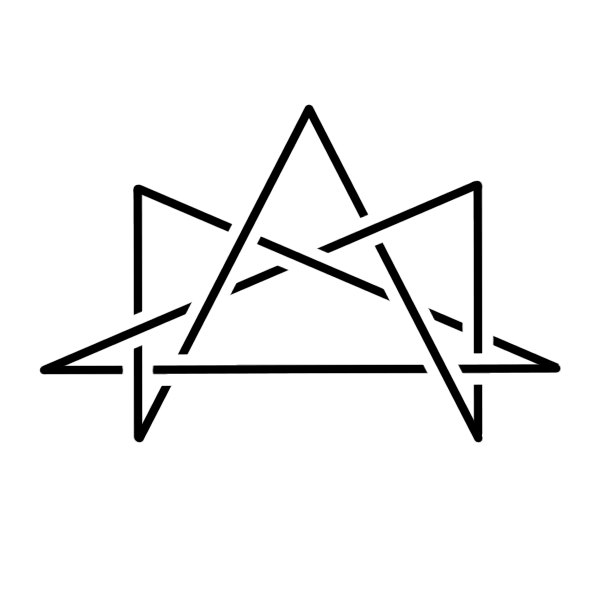}}\\
        \(8_8\)  & 6 & 5 & 9  & \parbox[c]{7em}{\includegraphics[width=7em]{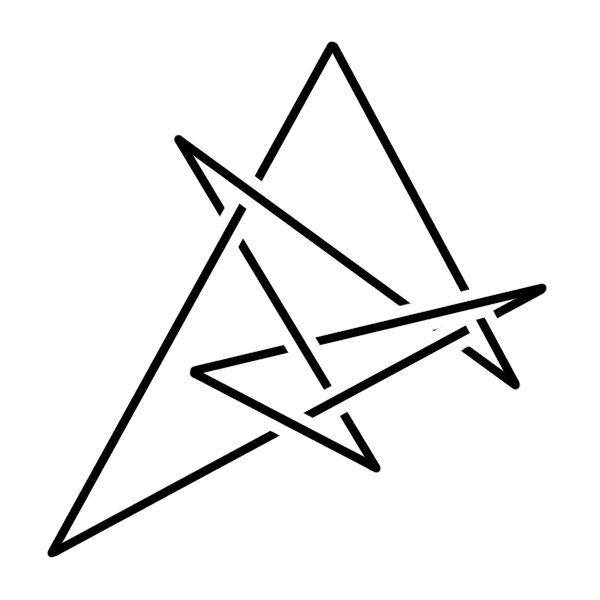}}	
        & \(9_7\)  & 6 & 5 & 9  & \parbox[c]{7em}{\includegraphics[width=7em]{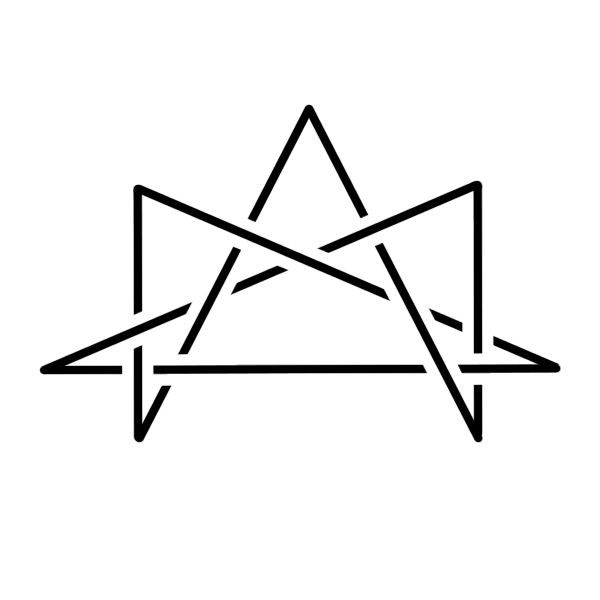}}\\
        \(8_{19}\)  & 6 & 7 & 7  & \parbox[c]{7em}{\includegraphics[width=7em]{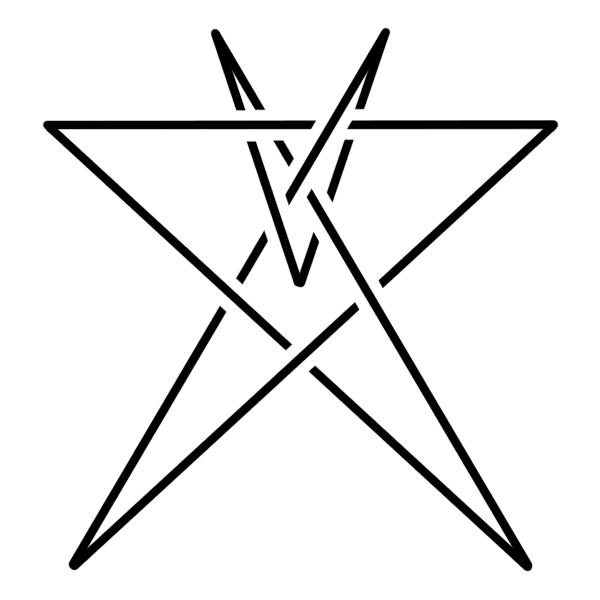}}	
        & \(9_9\)  & 6 & 5 & 9  & \parbox[c]{7em}{\includegraphics[width=7em]{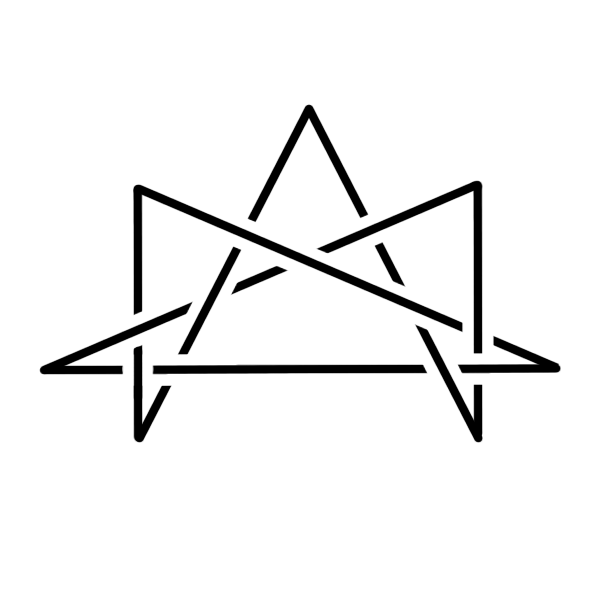}}\\ 
        \(8_{20}\)  & 6 & 7 & 7  & \parbox[c]{7em}{\includegraphics[width=7em]{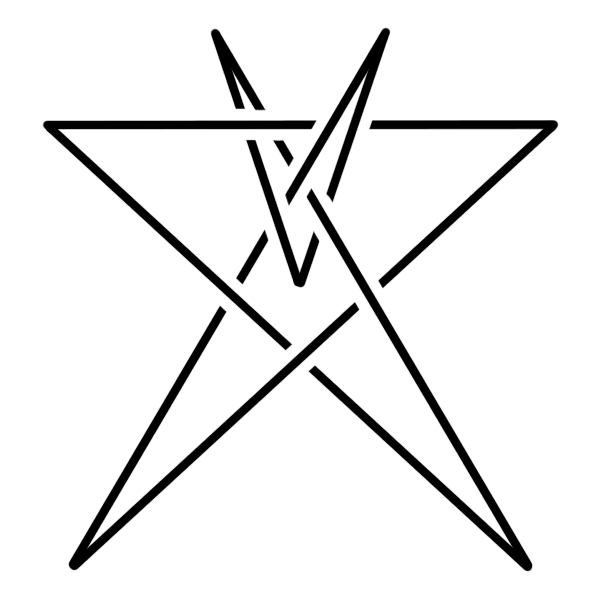}}	
        & \(9_{11}\)  & 6 & 5 & 9  & \parbox[c]{7em}{\includegraphics[width=7em]{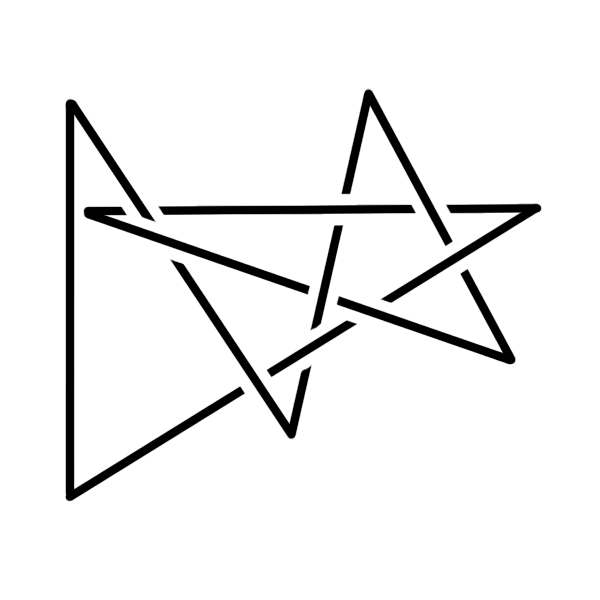}}\\
        \(9_1\)  & 6 & 5 & 9  & \parbox[c]{7em}{\includegraphics[width=7em]{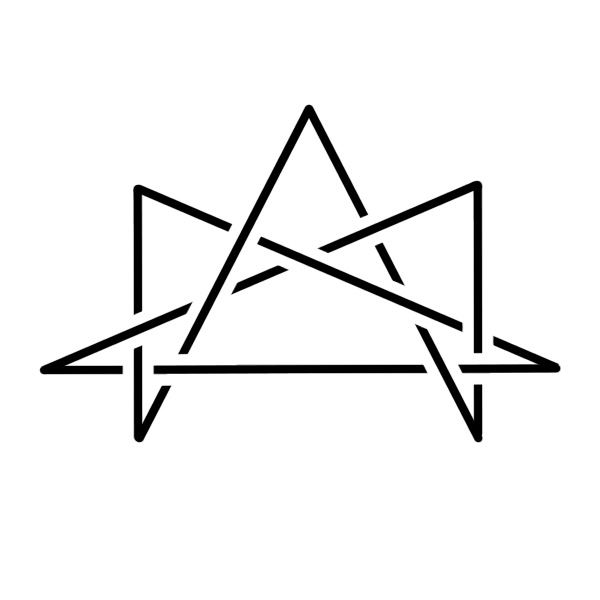}}
    & \(9_{20}\)  & 6 & 5 & 9  & \parbox[c]{7em}{\includegraphics[width=7em]{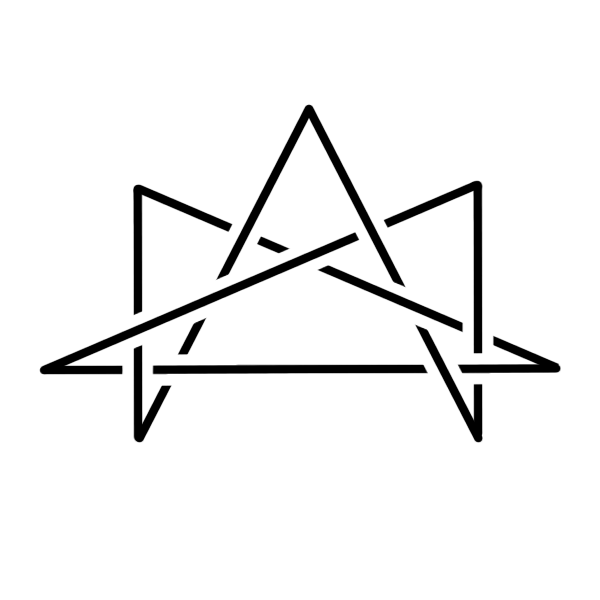}}\\
   &&&&& \(9_{26}\)  & 6 & 5 & 9  & \parbox[c]{7em}{\includegraphics[width=7em]{9-20.png}}
    \end{tabular}
\end{minipage}

\end{document}